\documentclass[12pt,final]{amsart}

\usepackage{mathrsfs}
\usepackage{txfonts}
\usepackage{amssymb,dsfont}
\usepackage{enumerate}
\usepackage{slashed}
\usepackage{fancyhdr}
\usepackage{aliascnt}
\usepackage{hyperref}
\usepackage[square,super,sort,longnamesfirst]{natbib}
\usepackage[dvips,dvipsnames]{xcolor}
\usepackage[dvips]{graphicx}
\usepackage[all,cmtip]{xy}
\usepackage{pdfsync}

\makeatletter
\newcommand{\autorefcheckize}[1]{%
  \expandafter\let\csname @@\string#1\endcsname#1%
  \expandafter\DeclareRobustCommand\csname relax\string#1\endcsname[1]{%
    \csname @@\string#1\endcsname{##1}\wrtusdrf{##1}}%
  \expandafter\let\expandafter#1\csname relax\string#1\endcsname
}
\makeatother

\hypersetup{
linkcolor=red,citecolor=green,filecolor=magenta,urlcolor=cyan}

\newcommand{\abs}[1]{\left\lvert#1\right\rvert}
\newcommand{\set}[1]{\left\{#1\right\}}
\newcommand{\field}[1]{\mathbb{#1}}
\newcommand{\R}{\field{R}}
\newcommand{\Lg}[1]{\mathrm{#1}}
\newcommand{\dif}{\mathrm d}
\newcommand{\diff}{\,\mathrm d}

\newcommand{\To}{\longrightarrow}
\newcommand{\Rmn}[1]{\uppercase\expandafter{\romannueral#1}}

\DeclareMathOperator{\Div}{div}

\DeclareMathOperator{\dist}{dist}

\DeclareMathOperator{\vol}{Vol}

\theoremstyle{plain}
\newtheorem{theorem}{Theorem}[section]

\newaliascnt{lem}{theorem}
\newtheorem{lem}[lem]{Lemma}
 \aliascntresetthe{lem}

\newaliascnt{cor}{theorem}
\newtheorem{cor}[cor]{Corollary}
 \aliascntresetthe{cor}

\newaliascnt{prop}{theorem}
\newtheorem{prop}[prop]{Proposition}
 \aliascntresetthe{prop}

\theoremstyle{remark}
\newtheorem{rem}{Remark}[section]

\theoremstyle{definition}
\newtheorem{defn}{Definition}[section]

\numberwithin{equation}{section}

\begin{document}

\title[quasi-harmonic]{A note on the nonexistence of quasi-harmonic spheres}


\author{Jiayu Li}
\address{School of Mathematics Sciences  \\ University of Science and Technology of China\\ 230026 Hefei, Anhui, China}
\email{jiayuli@ustc.edu.cn}

\author{Linlin Sun}
\address{School of Mathematics Sciences  \\ University of Science and Technology of China\\ 230026 Hefei, Anhui, China}
\email{sunll@ustc.edu.cn}

\thanks{The authors were supported in part by NSF in China (No. 11571332, 11131007, 11526212, 11426236). 
The authors would like to thank ZHU Xiangrong for his  useful discussions. }

\date{\today}

\begin{abstract}%
In this paper we study the properties of quasi-harmonic spheres from $\R^m, m>2$. We show that if the universal covering $\tilde N$ of $N$ admits a nonnegative
strictly convex function $\rho$ with the exponential growth condition $\rho(y)\leq C\exp\left(\frac14\tilde d(y)^{2/m}\right)$ where $\tilde d(y)$ is the distance function on $\tilde N$,
then $N$ does not admit a quasi-harmonic sphere, which generalize Li-Zhu's result \cite{Li2010non}. 
We also show that if $u$ is a quasi-harmonic sphere,
then the property that $u$ is of finite energy ($\int_{\R^m}e(u)e^{-\abs{x}^2/4}\dif x<\infty$) is equivalent to the property that $u$ satisfies the large energy condition ($\lim_{R\to\infty}R^{m}e^{-R^2/4}\int_{B_R(0)}e(u)e^{-\abs{x}^2/4}\diff x=0$).
\end{abstract}%

\maketitle

\section{Introduction}
Let $M^m,N^n$ be two compact Riemannian manifolds of dimension $m$ and $n$ respectively. Let $u\in W^{1,2}(M,N)$, the energy of $u$ is defined by
\begin{align*}
E(u)=\dfrac12\int_M\abs{\dif u}^2\diff\vol_M.
\end{align*}
The critical points of the energy functional are called harmonic maps.
Eells and Sampson \cite{Eells1964harmonic} introduce the  heat flow and prove that, the heat flow has a global solution which subconverges strongly to a harmonic map at infinity if the sectional curvature of the target manifold is non-positive. This result was generalized by Ding and Lin \cite{Ding1992generalization} to the case that the universal covering of $N$ admits a nonnegative strictly convex function with quadratic growth.
\par
However, in general, the heat flow may produce singularities at a finite time (e.g. \cite{Chang1992finite,Coron1989explosion}). Struwe divided singularities of the heat flow into two different types. One of this type is associated to quasi-harmonic spheres (c.f. \cite{Lin1999harmonic}).
\par

\begin{defn}
A quasi-harmonic sphere is a harmonic map from $\left(\R^m,\exp(-x^2/2(m-2))g_0\right)$ to a Riemannian manifold, where $g_0$ is the Euclidean metric in $\R^m$ ($m>2$), i.e.,
\begin{equation}\label{eq:quasi}
\tau(u)=\dfrac12x\cdot\dif u,
\end{equation}
with finite energy
\begin{equation}\label{eq:energy}
\int_{\R^m}e(u)e^{-\abs{x}^2/4}\dif x<\infty,
\end{equation}
where
\begin{equation*}
e(u)=\dfrac12\abs{\dif u}^2.
\end{equation*}
\end{defn}
Based on the work of Lin and Wang \cite{Lin1999harmonic}, we know that Liouville theorems for harmonic spheres (harmonic maps from spheres) and quasi-harmonic spheres imply the global existence of the heat flows. Li and Wang \cite{Li2009liouville} proved that there are no non-constant quasi-harmonic spheres with images in a regular ball.  Li and Zhu \cite{Li2010non} proved that, if the heat flow has a global solution and there is no harmonic map from $\Lg{S}^{l}$ to $N$ for $2\leq l\leq m-1$, then this flow subconverges in $C^2$ norm to a smooth harmonic map at infinity. Moreover, in the same paper, they also proved that the heat flow exists globally provided that the universal covering $\tilde N$ of $N$ admits a strictly convex positive function $\rho$ with polynomial growth, i.e.,
\begin{align*}
\tilde\nabla^2\rho>0,\quad0<\rho(y)<C(1+\tilde d(y,y_0))^P,\quad\forall y\in\tilde N,
\end{align*}
for some $y_0\in\tilde N$ and some positive constants $C,P$. Here $\tilde d$ is the distance function on $\tilde N$. Li and Yang \cite{Li2012nonexistence} generalized these results to the case of ``quasi-harmonic sphere with large energy condition" under the same assumption on $\rho$. The large energy condition is defined by
\begin{align}\label{eq:large}
\lim_{R\To\infty}R^{m}e^{-R^2/4}\int_{B_{R}(0)}e(u)e^{-\abs{x}^2/4}\diff x=0.
\end{align}
\par
Our first main result is as follows.
\begin{theorem}\label{thm:main1}Suppose $u$ satisfies \eqref{eq:quasi}, then the following three conditions are equivalent to each other.
\begin{enumerate}
\item The large energy condition holds, i.e., \eqref{eq:large} holds.
\item 
\begin{align*}
\int_{\R^m}\abs{u_r}^2\abs{x}^{4-m}\diff x<\infty.
\end{align*}
\item The total energy is finite, i.e., \eqref{eq:energy} holds.
\end{enumerate}
\end{theorem}
\begin{rem}Li and Zhu \cite{Li2010non} stated the following estimate for quasi-harmonic sphere,
\begin{align}\label{eq:scaling}
\int_{B_R(0)}\abs{\dif u}^2\diff x\leq CR^{m-2},\quad\forall R>0,
\end{align}
where $C$ is a constant independent of $R$. As a consequence, this condition \eqref{eq:scaling}\footnote{We thank ZHU Xiangrong for pointing out this equivalent condition.} is equivalent to \eqref{eq:energy} and is also equivalent to the following condition
\begin{align*}
\int_{\R^m}\abs{\dif u}^2\abs{x}^{2-m-\delta}\diff x<\infty
\end{align*}
for some or every $\delta>0$. In fact, one can get more, see \autoref{lem:energy}.
\end{rem}
Our second main result is that, Li-Zhu's result holds, if the universal covering $\tilde N$ of $N$ admits a nonnegative strictly convex function $\rho$ with the following exponential  growth condition: for some constant $C$,
\begin{align}\label{eq:growth}
\rho(y)\leq C\exp\left(\dfrac14\tilde d(y)^{2/m}\right),\quad\forall y\in\tilde N.
\end{align}
Here $\tilde d(y)=\tilde d(y,y_0)$ is the distance function on $\tilde N$ from some fixed point $y_0\in\tilde N$.
It is easy to check that this assumption is weaker than the one in \cite{Li2010non}.
\begin{theorem}\label{thm:main2}
Suppose $m\geq 3$ and there is a nonnegative strictly convex function  $\rho$ on the universal covering of the target manifold $N$ such that \eqref{eq:growth} holds. Then  there is no non-constant quasi-harmonic sphere $u$ from $\R^m$ to $N$.
\end{theorem}
\section{Proof of \autoref{thm:main1}}
In this section, we derive some estimates and prove \autoref{thm:main1}. Introduce 
\begin{align*}
H(r)\coloneqq\int_{\Lg{S}^{m-1}}\left(\abs{u_r}^2-e(u)\right)\dif\theta,\quad\forall r>0.
\end{align*}
We begin with the following Lemma.
\begin{lem}\label{lem:basic1}Suppose $u$ satisfies \eqref{eq:quasi}. 
Then
\begin{enumerate}
\item either
\begin{equation}\label{eq:2}
-R^{-2}(m-2)\int_{B_{\sqrt{2(m-2)}}}r^{2-m}\abs{u_r}^2\diff x\leq H(R)\leq0,\quad\forall R>0,
\end{equation}
\item or there exists  $R_0\geq\sqrt{2(m-2)}$  such that
\begin{equation}\label{eq:22}
H(R)\geq R^{2-2m}e^{R^2/2}R_0^{2m-2}e^{-R_0^2/2}H(R_0)>0,\quad\forall\ R>R_0.
\end{equation}
\end{enumerate}
Here $\Lg{S}^{m-1}$ stands for the unit sphere in $\R^m$ centering at $0$ and $B_R=B_R(0)$.
\end{lem}
\begin{proof}A direct computation gives (c.f. Lemma 3.3 in \cite{Li2010non})
\begin{equation}\label{eq:basic}
\dfrac{\dif}{\dif r}\int_{\Lg{S}^{m-1}}\left(\abs{u_r}^2-e(u)\right)\dif\theta-\int_{\Lg{S}^{m-1}}\left(\dfrac2re(u)+\left(\dfrac{r}{2}-\dfrac{m}{r}\right)\abs{u_r}^2\right)\dif\theta=0,\quad \forall r>0.
\end{equation}
According to this identity, we get
\begin{align*}
\dfrac{\dif}{\dif r}\int_{\Lg{S}^{m-1}}\left(\abs{u_r}^2-e(u)\right)\dif\theta+\dfrac{2}{r}\int_{\Lg{S}^{m-1}}\left(\abs{u_r}^2-e(u)\right)\dif\theta=\left(\dfrac{r}{2}-\dfrac{m-2}{r}\right)\int_{\Lg{S}^{m-1}}\abs{u_r}^2\diff\theta.
\end{align*}
From this formula, we know
\begin{align}\label{eq:H20}
\dfrac{\dif}{\dif r}\left(r^2H(r)\right)=r^2\left(\dfrac{r}{2}-\dfrac{m-2}{r}\right)\int_{\Lg{S}^{m-1}}\abs{u_r}^2\diff\theta.
\end{align}
Thus, $r^2H(r)$ is increase from $\sqrt{2(m-2)}$ to infinity, and is decrease from $0$ to $\sqrt{2(m-2)}$. Setting $C_0\coloneqq\sqrt{2(m-2)}$, we get
\begin{align*}
r^2H(r)\geq C_0^2H(C_0),\quad\forall r>0.
\end{align*}
Again according to \eqref{eq:basic} to obtain
\begin{align*}
&\dfrac{\dif}{\dif r}\int_{\Lg{S}^{m-1}}\left(\abs{u_r}^2-e(u)\right)\dif\theta+\left(\dfrac{2(m-1)}{r}-r\right)\int_{\Lg{S}^{m-1}}\left(\abs{u_r}^2-e(u)\right)\dif\theta\\
=&\left(r-\dfrac{2(m-2)}{r}\right)\int_{\Lg{S}^{m-1}}\left(e(u)-\dfrac12\abs{u_r}^2\right)\diff\theta,
\end{align*}
which implies
\begin{equation}\label{eq:basic1}
\dfrac{\dif}{\dif r}\left(r^{2m-2}e^{-r^2/2}H(r)\right)= r^{2m-2}e^{-r^2/2}\left(r-\dfrac{2m-4}{r}\right)\int_{\Lg{S}^{m-1}}\left(e(u)-\dfrac12\abs{u_r}^2\right)\diff\theta.
\end{equation}
Hence, $r^{2m-2}e^{-r^2/2}H(r)$ is increase from $\sqrt{2(m-2)}$ to infinity, and is decrease from $0$ to $\sqrt{2(m-2)}$. It is obvious that
\begin{equation*}
r^{2m-2}e^{-r^2/2}\int_{\Lg{S}^{m-1}}\left(\abs{u_r}^2-e(u)\right)\dif\theta\to0,\quad\text{as}\  r\to0.
\end{equation*}
\par

Moreover, 
\begin{align*}
\dfrac{\dif}{\dif r}\left(r^2H(r)\right)\geq-(m-2)r\int_{\Lg{S}^{m-1}}\abs{u_r}^2\diff\theta,
\end{align*}
which yields
\begin{align*}
R^2H(R)\geq-(m-2)\int_{B_R}r^{2-m}\abs{u_r}^2\diff x,\quad\forall R>0.
\end{align*}
Here we have used the fact
\begin{align*}
\lim_{r\to0}r^2H(r)=0.
\end{align*}
Therefore,
\begin{align*}
r^2H(r)\geq C_0^2H(C_0)\geq-(m-2)\int_{B_{C_0}}r^{2-m}\abs{u_r}^2\diff x,\quad\forall r>0.
\end{align*}
\par
Now we can finish the  proof of  this Lemma. If we do not have \eqref{eq:2}, then there exists $R_0\geq\sqrt{2(m-2)}$, such that
\begin{align*}
\int_{\set{R_0}\times\Lg{S}^{m-1}}\left(\abs{u_r}^2-e(u)\right)\dif\theta>0,
\end{align*}
then for every $r>R_0$,
\begin{align*}
&r^{2m-2}e^{-r^2/2}H(r)\geq R_0^{2m-2}e^{-R_0^2/2}H(R_0)>0,
\end{align*}
which means that \eqref{eq:22} holds.
\end{proof}
\begin{rem}Suppose $u$ satisfies \eqref{eq:quasi}, then 
\begin{align}
-R^2H(R)\leq&(m-2)\int_{B_{\sqrt{2(m-2)}}}r^{2-m}\abs{u_r}^2\diff x,\label{eq:H0}\\
-R^{2m-2}e^{-R^2/2}H(R)\leq&(m-2)\int_{B_{\sqrt{2(m-2)}}}r^{m-2}e^{-r^2/2}\dfrac{\abs{u_{\theta}}^2}{r^2}\diff x \label{eq:H2},\\
-R^{m}e^{-R^2/4}H(R)\leq&(m-2)\int_{B_{\sqrt{2(m-2)}}}e^{-r^2/4}e(u)\diff x \label{eq:H4},
\end{align}
holds for all $R>0$.
\end{rem}
\begin{proof}The proof of \eqref{eq:H0} and \eqref{eq:H2} can be found in the proof of \autoref{lem:basic1}. The proof of \eqref{eq:H4} can be proved similarly since \eqref{eq:basic} implies the following formula
\begin{equation*}
\dfrac{\dif}{\dif r}\left(r^me^{-r^2/4}H(r)\right)=\left(\dfrac{r}{2}-\dfrac{m-2}{r}\right)r^me^{-r^2/4}\int_{\Lg{S}^{m-1}}e(u)\dif\theta,\quad \forall r\in(0,\infty).
\end{equation*}
\end{proof}
\begin{lem}\label{lem:1.2}Suppose $u$ satisfies \eqref{eq:quasi} and
\begin{equation*}
\liminf_{R\to\infty}R^{2m-2}e^{-R^2/2}\int_{\set{R}\times\Lg{S}^{m-1}}\left(\abs{u_r}^2-e(u)\right)\dif\theta>0,
\end{equation*}
then
\begin{align*}
\liminf_{R\to\infty}R^{m}e^{-R^2/4}\int_{B_R}\left(\abs{u_r}^2-e(u)\right)e^{-r^2/4}\diff x>0.
\end{align*}
\end{lem}
\begin{proof}A direct computation.
\end{proof}

Next, we prove the following energy estimate.
\begin{prop}\label{lem:C5}Suppose $u$ satisfies \eqref{eq:quasi}, then there is a constant $C_1$ depending only on $m$ such that for every $0\leq\delta\leq2$, we have
\begin{align*}
\int_{B_R}r^{4-m-\delta}\abs{u_r}^2\diff x\leq C_1\int_{B_{2\sqrt{m-2}}}r^{2-m}\abs{u_r}^2\diff x+4R^2H(R)^+,\quad\forall R>0.
\end{align*}
Here $f^+=\max\set{f,0}.$
\end{prop}
\begin{proof}We only consider the case $R>2\sqrt{(m-2)}$ and start with the formula \eqref{eq:H20}, i.e.,
\begin{align*}
\dfrac{\dif}{\dif r}\left(r^2H(r)\right)=r^2\left(\dfrac{r}{2}-\dfrac{m-2}{r}\right)\int_{\Lg{S}^{m-1}}\abs{u_r}^2\diff\theta.
\end{align*}
 For every $0<\rho<R$, we have 
\begin{align*}
R^2H(R)-\rho^2H(\rho)=&\int_{\rho}^Rr^2\left(\dfrac{r}{2}-\dfrac{m-2}{r}\right)\int_{\Lg{S}^{m-1}}\abs{u_r}^2\diff\theta\diff r\\
=&\int_{B_R\setminus B_{\rho}}\left(\dfrac r2-\dfrac{m-2}{r}\right)r^{3-m}\abs{u_r}^2\diff x.
\end{align*}
For $\sqrt{4(m-2)}\leq\rho<R$, we have
\begin{align*}
\int_{B_R\setminus B_{\rho}}r^{4-m}\abs{u_r}^2\diff x\leq 4R^{2}H(R)^+-4\rho^{2}H(\rho),
\end{align*}
which implies
\begin{align*}
\int_{B_R\setminus B_{2\sqrt{m-2}}}r^{4-m}\abs{u_r}^2\diff x\leq& 4R^{2}H(R)^+-4\left(2\sqrt{m-2}\right)^{2}H\left(2\sqrt{m-2}\right)\\
\leq&4R^{2}H(R)^++4(m-2)\int_{B_{2\sqrt{m-2}}}r^{2-m}\abs{u_r}^2\diff x.
\end{align*}
Here we have used \eqref{eq:H0}. In particular, we get the desired estimate for $\delta=0$. In general $0\leq\delta\leq2$,
\begin{align*}
\int_{B_R}r^{4-m-\delta}\abs{u_r}^2\diff x=&\int_{B_R\setminus B_{2\sqrt{m-2}}}r^{4-m-\delta}\abs{u_r}^2\diff x+\int_{B_{2\sqrt{m-2}}}r^{4-m-\delta}\abs{u_r}^2\diff x\\
\leq&\int_{B_R\setminus B_{2\sqrt{m-2}}}r^{4-m}\abs{u_r}^2\diff x+\left(2\sqrt{m-2}\right)^{2-\delta}\int_{B_{2\sqrt{m-2}}}r^{2-m}\abs{u_r}^2\diff x\\
\leq&8(m-2)\int_{B_{2\sqrt{m-2}}}r^{2-m}\abs{u_r}^2\diff x+4R^2H(R)^+.
\end{align*}
\end{proof}
As a consequence, 
\begin{cor}\label{lem:morrey}Suppose $u$ satisfies \eqref{eq:quasi}. Then there is a constant $C_2$ such that for every $0<\delta<1$,
\begin{align*}
\delta R^{-\delta}\int_{B_R}r^{2-m+\delta}e(u)\diff x\leq&C_2\int_{B_{2\sqrt{m-2}}}r^{2-m}\abs{u_r}^2\diff x+4R^2H(R)^+,\quad\forall R>0.
\end{align*}
In particular,
\begin{align}\label{eq:morrey}
R^{2-m}\int_{B_R}e(u)\diff x\leq&C_2\int_{B_{2\sqrt{m-2}}}r^{2-m}\abs{u_r}^2\diff x+4R^2H(R)^+,\quad\forall R>0.
\end{align}
\end{cor}
\begin{proof}Since
\begin{align*}
\int_{B_R}r^{2-m+\delta}e(u)\diff x=&-\int_0^Rr^{1+\delta}H(r)\diff r+\int_{B_R}r^{2-m+\delta}\abs{u_r}^2\diff x\\
\leq&\sup_{0<r<R}\left(-r^2H(r)\right)\times\int_0^Rr^{\delta-1}\diff r+R^{\delta}\int_{B_R}r^{2-m}\abs{u_r}^2\diff x\\
=&\sup_{0<r<R}\left(-r^2H(r)\right)\times\dfrac{R^{\delta}}{\delta}+R^{\delta}\int_{B_R}r^{2-m}\abs{u_r}^2\diff x.
\end{align*}
Now applying \autoref{lem:basic1} and \autoref{lem:C5}, there exists a constant $C_2$ depending only on $m$ such that
\begin{align*}
\delta R^{-\delta}\int_{B_R}r^{2-m+\delta}e(u)\diff x\leq&C_2\int_{B_{2\sqrt{m-2}}}r^{2-m}\abs{u_r}^2\diff x+4R^2H(R)^+.
\end{align*}
\end{proof}
Also, we can prove the following 
\begin{cor}\label{lem:energy}Suppose $u$ satisfies \eqref{eq:quasi}, then  there is a constant $C_3$ depending only on $m$ such that  for every  $0<\delta<1$,
\begin{align*}
\delta\int_{B_R}r^{2-m-\delta}e(u)\diff x\leq&C_3\int_{B_{2\sqrt{m-2}}}r^{1-m}e(u)\diff x+4R^2H(R)^{+},\quad\forall R>0.
\end{align*}
\end{cor}
\begin{proof}Similar to the proof of \autoref{lem:morrey}, for $0<\delta<1$ and $R>2\sqrt{m-2}$,
\begin{align*}
\int_{B_R\setminus B_{2\sqrt{m-2}}}r^{2-m-\delta}e(u)\diff x=&-\int_{2\sqrt{m-2}}^Rr^{1-\delta}H(r)\diff r+\int_{B_R\setminus B_{2\sqrt{m-2}}}r^{2-m-\delta}\abs{u_r}^2\diff x\\
\leq&\sup_{2\sqrt{m-2}<r<R}\left(-r^2H(r)\right)\times\int_{2\sqrt{m-2}}^Rr^{\delta-1}\diff r+\int_{B_R\setminus B_{2\sqrt{m-2}}}r^{2-m}\abs{u_r}^2\diff x\\
\leq&\sup_{2\sqrt{m-2}<r<R}\left(-r^2H(r)\right)\times\dfrac{2\sqrt{m-2}}{\delta}+\int_{B_R}r^{2-m}\abs{u_r}^2\diff x.
\end{align*}
Then \autoref{lem:basic1} and \autoref{lem:C5} gives the desired estimate.
\end{proof}

Now we prove \autoref{thm:main1}.

\begin{proof}[Proof of \autoref{thm:main1}]
Suppose the large energy condition holds, i.e., the claim $(1)$ is true. Then according to \autoref{lem:basic1} and \autoref{lem:1.2} (or c.f. \cite{Li2012nonexistence}), we know that $H(r)\leq0$ for every $r>0$. Now the claim $(2)$ follows from \autoref{lem:C5}. 
\par
From the claim $(2)$ to the claim $(3)$, we need only to prove that
\begin{align*}
\int_{\R^m}r^{2-m-\delta}\abs{\dif u}^2\diff x<\infty.
\end{align*}
holds for some $\delta>0$. According to \autoref{lem:energy}, we need only to claim that $\liminf_{R\to\infty}R^2H(R)^{+}\leq0$. This is true because
\begin{align*}
\liminf_{R\to\infty}R^2H(R)^{+}\leq\liminf_{R\to\infty}\int_{\set{R}\times\Lg{S}^{m-1}}\abs{u_r}^2\diff\theta
\end{align*}
and the claim $(2)$ implies the righthand is zero.
\par
From the claim $(3)$ to the claim $(1)$ is obvious.
\end{proof}

\section{Proof of \autoref{thm:main2}}
The following Lemma is proved in \cite{Li2010non}. Here we provide another proof which is simpler for $m>2$.
\begin{lem}\label{lem:decay}Suppose $f$ is a non-constant nonnegative smooth function satisfying
\begin{equation*}
\Delta f\geq\dfrac12rf_r,
\end{equation*}
then there exists a constant $C>0$ such that for $r$ large enough,
\begin{equation*}
\int_{\Lg{S}^{m-1}}f(r,\theta)\diff\theta>Cr^{-m}e^{r^2/4}.
\end{equation*}
\end{lem}
\begin{proof}Let
\begin{align*}
v(r)=\int_{\Lg{S}^{m-1}}f(r,\theta)\diff\theta,
\end{align*}
then a direct computation yields
\begin{align*}
\dfrac{\dif}{\dif r}\left(r^{m-1}e^{-r^2/4}\dfrac{\dif}{\dif r}v\right)\geq0.
\end{align*}
Since $\tfrac{\dif v}{\dif r}=O\left(\tfrac1r\right)$ as $r\to 0$, we obtain
\begin{align*}
\lim_{r\to0}r^{m-1}e^{-r^2/4}\dfrac{\dif}{\dif r}v=0,
\end{align*}
since $m>2$. In particular,
\begin{align*}
r^{m-1}e^{-r^2/4}\dfrac{\dif}{\dif r}v\geq0.
\end{align*}
Since $f$ is not a constant, there exists $a>0$ such that $\frac{\dif v}{\dif r}\vert_a>0$. The rest of the proof is simple (c.f. \cite{Li2010non}).
\end{proof}

Let $d(x)=\dist(u(x),u(0))$, then we have the following
\begin{lem}[Refine energy estimate]\label{lem:energy-estimate}Suppose $u$ is a quasi-harmonic sphere, then there is a constant $C_m$ depending only on $m$ such that for all $R>0$,
\begin{align*}
\int_{B_R}d^2\diff x\leq& C_mR^m \int_{B_{2\sqrt{m-2}}}r^{1-m}\abs{u_r}^2\diff x,\\
\int_{B_R}\abs{\nabla d}^2\diff x\leq& C_mR^{m-2}\int_{B_{2\sqrt{m-2}}}r^{1-m}\abs{u_r}^2\diff x.
\end{align*}
\end{lem}
\begin{rem}
\begin{enumerate}
\item Denoted $E_R(u)$ by the energy of $u$ on $B_R$, i.e.,
\begin{align*}
E_R(u)=\dfrac{1}{2}\int_{B_R}\abs{\dif u}^2e^{-x^2/4}\diff x.
\end{align*}
Then apply \autoref{lem:energy} to this Lemma to obtain the following estimate
\begin{align*}
\int_{B_R}d^2\diff x\leq& C_mR^m E_{R}(u),\\
\int_{B_R}\abs{\nabla d}^2\diff x\leq& C_mR^{m-2}E_{R}(u).
\end{align*}
\item
Li and Zhu (c.f. Lemma 3.2 in \cite{Li2010non}) obtained a similar result with constant $C_{m,u}$ depending only on $m$ and the total energy of $u$ such that
\begin{align*}
\int_{B_R}d^2\diff x\leq& C_{m,u}R^m,\\
\int_{B_R}\abs{\nabla d}^2\diff x\leq& C_{m,u}R^{m-2}.
\end{align*}

\end{enumerate}
\end{rem}
\begin{proof}[Proof of \autoref{lem:energy-estimate}]It is clear that
\begin{align*}
d(r,\theta)\leq\int_0^r\abs{u_s(s,\theta)}\diff s,\quad\abs{\nabla d}\leq\abs{\dif u}.
\end{align*}
Since the total energy of $u$ is finite, by \autoref{lem:1.2}, we have
\begin{align*}
\int_{\Lg{S}^{m-1}}\left(\abs{u_r}^2-e(u)\right)\diff\theta\leq0,\quad r>0.
\end{align*}
Applying \eqref{eq:morrey}, we obtain
\begin{equation*}
\int_{B_R}\abs{\nabla d}^2\leq 2C_2R^{m-2}\int_{B_{2\sqrt{m-2}}}r^{2-m}\abs{u_r}^2\diff x,\quad R>0.
\end{equation*}
\par
Next, we show
\begin{equation*}
\int_{\Lg{S}^{m-1}}\left(\int_0^r\abs{u_{s}(s,\theta)}\diff s\right)^2\diff\theta\leq C_m\int_{B_{2\sqrt{m-2}}}r^{1-m}\abs{u_r}^2\diff x,\quad\forall r>0.
\end{equation*}
Then the first part of the this Lemma follows from this inequality.
Without loss of generality, assume $r>1$. Applying \autoref{lem:C5} and taking $\delta=1/2$, we get
\begin{align*}
\int_{B_R}r^{7/2-m}\abs{u_r}^2\diff x\leq C_1\int_{B_{2\sqrt{m-2}}}r^{2-m}\abs{u_r}^2\diff x,\quad R>0.
\end{align*}
Using Minkowski's inequality, we get
\begin{align*}
\left(\int_{\Lg{S}^{m-1}}\left(\int_0^r\abs{u_{s}(s,\theta)}\diff s\right)^2\diff\theta\right)^{1/2}\leq&\int_0^r\left(\int_{\Lg{S}^{m-1}}\abs{u_s(s,\theta)}^2\diff\theta\right)^{1/2}\diff s\\
\leq&\int_0^1\left(\int_{\Lg{S}^{m-1}}\abs{u_s(s,\theta)}^2\diff\theta\right)^{1/2}\diff s+\int_1^r\left(\int_{\Lg{S}^{m-1}}\abs{u_s(s,\theta)}^2\diff\theta\right)^{1/2}\diff s\\
\leq&\left(\int_0^1\int_{\Lg{S}^{m-1}}\abs{u_s(s,\theta)}^2\diff\theta\diff s\right)^{1/2}\\
&+\left(\int_1^rs^{5/2}\int_{\Lg{S}^{m-1}}\abs{u_s}^2\diff\theta\diff s\right)^{1/2}\left(\int_1^rs^{-5/2}\diff s\right)^{1/2}\\
\leq&C_m\left(\int_{B_{2\sqrt{m-2}}}r^{1-m}\abs{u_r}^2\diff x\right)^{1/2}.
\end{align*}
\end{proof}
\begin{lem}\label{lem:distance}Suppose $u$ is a quasi-harmonic sphere, then there is a constant $C_m$  depending only on $m$ such that
\begin{equation*}
\fint_{B_r}\exp\left(C_m^{-1}E_r(u)^{-1/2}r^{2-m}d\right)\diff x\leq C_m,\quad\forall r>1.
\end{equation*}
\end{lem}
\begin{proof}
By the energy estimate \autoref{lem:energy}, using an argument similar to the one used in the proof of Lemma 3.5 in \cite{Li2010non}, we can prove  that the BMO subnorm $[d]_{*,B_{2r}}$ of $d$ over $B_{2r}$ satisfies
\begin{align}\label{eq:bmo}
[d]_{*,B_{2r}}\coloneqq\sup_{x\in Q\subset B_{2r}}\fint\abs{d(y)-d_{Q}}\diff y\leq C_{m}\sqrt{E_{2r}(u)}(1+r)^{m-2},
\end{align}
where the supermum is taken over all cubes $x\in Q\subset B_{2r}$. The John-Nirenberg theorem (c.f. Lemma 1 in \cite{John1961on}) claims that there is two constants $C_5,C_6$ depends only on $m$ such that for all cubes $Q\subset B_{2r}$,
\begin{align*}
\abs{\set{x\in Q:\abs{d(x)-d_Q}>s}}\leq C_5\exp\left(-\dfrac{C_6s}{[d]_{*,B_{2r}}}\right)\abs{Q},
\end{align*}
which implies
\begin{align*}
\fint_{B_{r}}\exp\left(\dfrac{C_6\abs{d-d_{B_r}}}{2[d]_{*,B_r}}\right)\diff x\leq  C_5,\quad\forall r>0.
\end{align*}
Since we have the estimate \eqref{eq:bmo}, as a consequence,  there is a constant $C_7$ which depends only on $m$ such that
\begin{align*}
\fint_{B_{r}}\exp\left(C_7^{-1}E_r(u)^{-1/2}r^{2-m}\abs{d-d_{B_r}}\right)\diff x\leq  C_7,\quad\forall r>1.
\end{align*}
Finally, according to \autoref{lem:energy-estimate}, we can find a constant $C_8$ depending only $m$ such that
\begin{align*}
d_{B_r}\coloneqq\fint_{B_r}d\diff x\leq C_8E_r(u)^{1/2}.
\end{align*}
Therefore, we get the desired estimate.
\end{proof}
\begin{rem}Checking the proof of Lemma 3.5 in \cite{Li2010non} step by step, and using the argument mentioned above, one can prove the following refine estimate,
\begin{align*}
\fint_{B_r}\exp\left(C_m^{-1}\tilde E_{2\sqrt{m-2}}(u)^{-1/2}r^{2-m}d\right)\diff x\leq C_m,\quad\forall r>1.
\end{align*}
Here
\begin{align*}
\tilde E_R(u)=\int_{B_R}r^{1-m}\abs{u_r}^2\diff x.
\end{align*}
In fact, checking the proof (c.f. page 455 in \cite{Li2010non} ), the constants come from either \autoref{lem:energy-estimate} or $\tilde E_{3m}(u)$ which can be controlled by $\tilde E_{2\sqrt{m-2}}(u)$ thanks to \autoref{lem:energy}. Hence one can prove the required refine BMO estimate \eqref{eq:bmo}.
\end{rem}
Now we give a poof of \autoref{thm:main2}.
\begin{proof}[Proof of \autoref{thm:main2}]Let $\tilde N$ be the universal covering of $N$. Let $\tilde u:\R^m\To\tilde N$ be a lift of $u$ with $\tilde u=u\circ\pi$ where $\pi:\tilde N\To N$ is the covering map. It is easy to see that
\begin{align*}
\int_{\R^m}e(\tilde u)e^{-\abs{x}^2/4}\diff x<\infty.
\end{align*}
Set $f=\rho\circ\tilde u$, then
\begin{align*}
\Delta f-\dfrac12r\partial_rf=\tilde\nabla^2\rho(\tilde u)(\dif\tilde u,\dif\tilde u)>0.
\end{align*}
\par
Fixed $p>0$. Notice that  there is a constant  $C>0$ such that
\begin{align}\label{eq:f1}
\int_{B_{2R}}f^p\diff x=\int_{B_{2R}}\left(\rho\circ\tilde u\right)^p\diff x\leq C^p\int_{B_{2R}}e^{\tfrac p4\tilde d^{2/m}}\diff x,\quad R>0.
\end{align}
Applying Young's inequality,
\begin{equation*}
A+B\geq\left(PA\right)^{1/P}\left(QB\right)^{1/Q},\quad A,B>0,\quad P,Q\geq1,\quad 1/P+1/Q=1,
\end{equation*}
we obtain that for $\tilde\delta=p/(2m)$,
\begin{align*}
\tilde\delta r^{2-m}\tilde d+\left(\dfrac p4-\tilde\delta\right)r^2=&\dfrac{p}{2m}r^{2-m}\tilde d+\left(\dfrac p4-\dfrac{p}{2m}\right)r^2\\
=&\dfrac{p}{4}\left(\dfrac{2}{m}r^{2-m}\tilde d+\dfrac{m-2}{m}r^2\right)\\
\geq&\dfrac{p}{4}\left( r^{2-m}\tilde d\right)^{2/m}\left(r^2\right)^{(m-2)/m}\\
=&\dfrac p4\tilde d^{2/m}.
\end{align*}
Therefore, according to \eqref{eq:f1}, for $R>0$, we have
\begin{align}\label{eq:f2}
\int_{B_{2R}}f^p\diff x\leq C^p\int_{B_{2R}}e^{\tilde\delta R^{2-m}\tilde d(\tilde u,y_0)}e^{(p/4-\tilde\delta)R^2}\diff x=C^p\int_{B_{2R}}e^{2^{m-2}\tilde\delta (2R)^{2-m}\tilde d(\tilde u,y_0)}e^{(p/4-\tilde\delta)R^2}\diff x.
\end{align}
We can choose $p>0$ sufficiently small so that
\begin{align*}
2^{m-2}\tilde\delta= 2^{m-3}m^{-1}p\leq C_m^{-1}E^{-1/2},
\end{align*}
which is equivalent to
\begin{equation*}
E\leq \dfrac{m^2}{4^{m-3}C_m^2p^2}.
\end{equation*}
According to \autoref{lem:distance} and \eqref{eq:f2}, we can see that
\begin{align*}
\int_{B_{2R}}f^p\diff x\leq& C^pe^{(p/4-\tilde\delta)R^2}\int_{B_{2R}}\exp\left(C_m^{-1}E^{-1/2}(2R)^{2-m}\tilde d(\tilde u,y_0)\right)\diff x\leq C^pC_m(2R)^me^{(p/4-p/(2m))R^2}
\end{align*}
holds for $R$ large enough.
\par
If $f$ is not a constant, applying \autoref{lem:decay} we obtain that for  $R$ large enough,
\begin{align*}
\int_{B_R}f\diff x\geq C_u R^{-2}e^{R^2/4}.
\end{align*}
Here $C_u>0$ is a constant which is independent of $R$.
Since $f\geq0$ satisfies
\begin{equation*}
\Div\left(e^{-\abs{x}^2/4}\nabla f\right)\geq0,
\end{equation*}
applying Moser's iteration (c.f. page 167 in \cite{Li2012nonexistence}), for every $p>0$, there is a constant $C_{p}>0$ depending only on $p,m$ such that
\begin{equation*}
\fint_{B_{R}}f\diff x\leq C_{p}R^{m/p}\left(\int_{B_{2R}}f^p\diff x\right)^{1/p}
\end{equation*}
holds for $R$ large enough. Consequently, for $R$ large enough
\begin{align}\label{eq:f3}
\int_{B_{2R}}f^p\diff x\geq C_{p}^{-p}C_u^pR^{-(m+2)p-m}e^{pR^2/4}.
\end{align}
Together with \eqref{eq:f2} and \eqref{eq:f3}, we know that
\begin{align*}
0<C_p^{-p}C_u^p\leq C^pC_m2^mR^{2m+(m+2)p}e^{-p R^2/(2m)}\to 0,\quad\text{as}\ R\to\infty.
\end{align*}
This contradiction  means that $f$ is a constant. Moreover, since $\rho$ is a strictly convex function, we get that $\dif\tilde u=0$, i.e., $\tilde u$ is a constant. As a consequence, $u$ is a constant.

\end{proof}


\end{document}